\documentclass[10pt,a4paper,final]{article}

\usepackage[latin1]{inputenc}
\usepackage{amsmath}
\usepackage{amsfonts}
\usepackage{comment}
\usepackage[margin=1.2in]{geometry}
\usepackage{amssymb}
\usepackage{amsthm}
\usepackage{makeidx}

\usepackage{graphicx}
\usepackage{todo}
\usepackage{latexsym}
\usepackage{cite}
\usepackage{mathrsfs}
\usepackage{bbm}
\usepackage{bbold}
\usepackage{float}
\usepackage{epsfig}
\usepackage{epstopdf}
\usepackage{hyperref}
\usepackage{mathtools}
\usepackage{listings}
\usepackage{pgf}
\usepackage[T1]{fontenc}
\usepackage[latin1]{inputenc}
\usepackage[english]{babel}
\usepackage{pifont,tabularx}
\usepackage{stmaryrd}
\usepackage{dsfont}
\usepackage{theoremref}
\usepackage{nicefrac}
\usepackage{enumerate}
\usepackage{comment}

\usepackage{tikz}
\usepackage{pgfplots}
\usepackage{pgfplotstable}
\pgfplotsset{compat=1.7}
\usepackage{caption}
\usepackage{subfigure}
\usepackage{caption}

\theoremstyle{plain} 
\newtheorem{thm}{Theorem}[section]

\newtheorem{lem}[thm]{Lemma}

\newtheorem{obs}[thm]{Remark}

\title{Very large cliques in a scale-free random graph}
\author{
Carlo~De~Ambroggio\thanks{Universit\'a degli studi di Torino, Dipartimento di Matematica Giuseppe Peano. Email: carlo.deambroggio@unito.it}\and Umberto~De~Ambroggio\thanks{National University of Singapore, Department of Mathematics. Email: umbidea@gmail.com}
}

\begin{document}

\maketitle

\begin{abstract}
    In this short article we consider a preferential attachment random graph model with edge steps, studied by Alves, Ribeiro and Sanchis. Starting with an initial graph $\mathbb{G}_1$ formed by a vertex with a self-loop attached to it, the model evolves as follows. At every subsequent (discrete) time step, either with probability $p$ we add a vertex to the graph and connect it to exactly one of the older vertices selected with probability proportional to its degree, or with probability $1-p$ we add one edge between two existing vertices, both selected (independently) with probability proportional to their degrees. Let $\omega(\mathbb{G})$ be the clique number of a graph $\mathbb{G}$, i.e.\ the number of vertices in a largest complete subgraph of $\mathbb{G}_{}$. Alves, Ribeiro and Sanchis showed that, for any given $\varepsilon>0$, we have $\omega(\mathbb{G}_{2t})\geq t^{\frac{1-p}{2-p}(1-\varepsilon)}$ with high probability (i.e.\ with probability tending to $1$ as $t\rightarrow \infty$). Here we strengthen this bound by showing that, for any  function $f:\mathbb{N}\mapsto \mathbb{N}$ that satisfies $f(t)\rightarrow \infty$ as $t\rightarrow \infty$, with high probability
    \[\omega(\mathbb{G}_{2t}) = \Omega\left(t^{\frac{1-p}{2-p}}\Big(\log^{\frac{1}{2-p}}(t)f(t)\Big)^{-1}\right).\] 
\end{abstract}

\section{Introduction}
Since its introduction, the preferential attachment model has attracted considerable interest due to its ability to capture key features observed in real-world networks, see e.g.\  \cite{van2017random} for a comprehensive introduction.

In this article, we focus on a preferential attachment model with \textit{edge steps} studied by Alves, Ribeiro and Sanchis \cite{alves2021preferential} and subsequently analyzed in \cite{alves2017large} and \cite{alves2021clustering}. Models of this type have also been studied in other works. For example, in \cite{cooper2004random} the authors studied the degree sequence of a more general version of the model considered here, while in \cite{chung2004coupling} the authors analyzed topological properties such as diameter and typical distances for the model introduced in \cite{cooper2004random}.

The model studied in \cite{alves2021preferential} is defined as follows. Starting with an initial graph $\mathbb{G}_1$ consisting of a vertex with a self-loop attached to it, the model evolves as follows. At every time step $t\geq 2$, one of the following operations is performed:
\begin{itemize}
    \item With probability $p$, a new vertex $v$ is added to the graph and connected to an existing node $u$, chosen with probability proportional to its degree. That is, letting $\mathbb{G}_t$ denote the graph at time $t$, 
    \[
    \mathbb{P}(v \to u \mid \mathbb{G}_t) = \frac{\deg(u,t)}{\sum_{w \in \mathbb{G}_t} \deg(w,t)}.
    \]
    
    \item With probability $1-p$, an edge is added between two existing vertices which are selected independently with probability proportional to their degrees. (We note that we allow loops to be
added, and we also allow a new connection to be added between vertices that already
shared an edge.)
\end{itemize}
Alves, Ribeiro and Sanchis \cite{alves2021clustering} showed that, for any given $\varepsilon >0$, as $t\rightarrow \infty$
\begin{equation}\label{upperboundclique}
    \mathbb{P}\Big(t^{\frac{1-p}{2-p}(1-\varepsilon)}\leq \omega(\mathbb{G}_{2t})\leq t^{\frac{1-p}{2-p}}\log^{C}(t)\Big)=1-o(1),
\end{equation}
where $\omega(\mathbb{G}_{})$ is the clique number of a graph $\mathbb{G}$ (i.e.\ the number of vertices in the largest complete subgraph of $\mathbb{G}$) and $C>0$ is some constant.
The purpose of this short article is to improve the lower bound on $\omega(\mathbb{G}_{2t})$ by showing that, up to poly-logarithmic corrections, with high probability (i.e.\ with probability tending to $1$ as $t \to \infty$), the random graph $\mathbb{G}_{2t}$ contains a clique of order $t^{\frac{1-p}{2-p}}$. More precisely, we prove the following
\begin{thm}\label{mainthm}
     Let $f:\mathbb{N}\mapsto \mathbb{N}$ be a function such that $f(t)\rightarrow \infty$ as $t\rightarrow \infty$, arbitrarily slowly. Then, with high probability, 
    \[\omega(\mathbb{G}_{2t})=\Omega\left(t^{\frac{1-p}{2-p}}\Big(\log^{\frac{1}{2-p}}(t)f(t)\Big)^{-1}\right).\]
\end{thm}
\begin{obs}
    We remark that Theorem \ref{mainthm} could be stated for $\mathbb{G}_{t}$; however, in this case, we would need to replace $t$ with $t/2$ throughout the proof, which is annoying. This is why we preferred (as in \cite{alves2021clustering}) to state the result for the random graph at time $2t$.
\end{obs}
\begin{obs}
Note that Theorem \ref{mainthm} together with the upper bound for $\omega(\mathbb{G}_{2t})$ displayed in (\ref{upperboundclique}) establishes that
\[\Big|\frac{\log(\omega(\mathbb{G}_t))}{\log(t)}- \frac{1-p}{2-p}\Big| =\Theta\Big(\frac{\log\log(t)}{\log(t)}\Big)\]
with high probability. It would be interesting to establish tight (up to constant factors) bounds for the clique number of this model. In particular, we can ask: is there a constant $b=b(p)\in \mathbb{R}$ such that
\[A^{-1}\log^{b}(t)t^{\frac{1-p}{2-p}}\leq \omega(\mathbb{G}_{2t})\leq A\log^{b}(t)t^{\frac{1-p}{2-p}}\]
with probability at least $1-1/g(A)$, where $g(A)\rightarrow \infty$ as $A\rightarrow \infty$?
\end{obs}

\paragraph{Notation.} We list here the notation used throughout the article. We denote by $\mathbb{N}$ the set of positive integers. Given $n\in \mathbb{N}$, we let $[n]\coloneqq \{1,2,\dots,n\}$. We use $\log(\cdot)$ for the natural logarithm and write $\log^\alpha(x)$ instead of the more precise $(\log(x))^{\alpha}$.
For two functions $f,g: \mathbb{R}_+ \to \mathbb{R}_+$ we write $f(x) = o(g(x))$ or $f(x) \ll g(x)$ to indicate that $ f(x)/g(x) \to 0$ as $ x \to \infty$; $f(x)=\Omega(g(x))$ if there exists a constant $c>0$ such that $f(x)\ge c g(x)$ as $x\to\infty$; and $f(x)=O(g(x))$ if there exists a constant $c>0$ such that $f(x)\le c g(x)$ as $x \to \infty$. We write $f(c)=\Theta(g(x))$ if $f(x)=O(g(x))$ and $g(x)=O(f(x))$. 
Given a graph $\mathbb{G}$, we write $V(\mathbb{G})$ and $E(\mathbb{G})$ to denote the set of vertices and edges of $\mathbb{G}$, respectively. Given a set $S$, we write $|S|$ to indicate its cardinality. 
 Given two random variables $X$ and $Y$, we write $X =_d Y$ if $X$ is a random variable distributed according to the distribution $Y$.
 We write $\text{Bin}(n,p)$ to indicate the binomial distribution with parameters $n,p$, while we use $\text{Exp}(\lambda)$ to indicate the exponential distribution with parameter $\lambda$. 
We write $\exp{(x)}$ to indicate the exponential function $e^{x}$. 
We write $\lceil x \rceil$ to indicate $\inf_{n \in \mathbb{N}}\{n : n \geq x\}$ and  $\lfloor x \rfloor$ to indicate $\sup_{n \in \mathbb{N}}\{n: n \leq x\}$.
Finally, we write $a \vee b$ to indicate the maximum between $a$ and $b$.

\subsection{Some known estimates}
Here we collect two basic estimates which will be needed later on. The first one is a classical concentration inequality for the binomial distribution, whereas the second one is the well-known Azuma-Hoeffding inequality for martingales with bounded increments evolving in discrete time.
\begin{lem}[Chernoff bound]\label{concbinom}
    Let $X\sim\mathrm{Bin}(n,p)$, and let $\mu = np$ . Then
\[
\mathbb{P}\left(X \le \frac{\mu}{2}\right) \le \exp\left(-\frac{\mu}{8}\right) .
\]
\end{lem}
\begin{lem}\label{concmg}
    Let $(X_t:t\in \mathbb{N}_0)$ be a martingale satisfying $|X_i-X_{i-1}|\leq b_i$ for each $1\leq i\leq n$. Then, given any $x>0$, we have that
    \[\mathbb{P}(X_n-X_0>x)\leq \exp\Big(-\frac{x^2}{2\sum_{i=1}^nb^2_i}\Big).\]
\end{lem}

\subsection{Heuristic derivation of Theorem \ref{mainthm}}
To prove our result we restrict the search of a clique of the desired size
\[t^{\frac{1-p}{2-p}}\Big(\log^{\frac{1}{2-p}}(t)f(t)\Big)^{-1}\]
among vertices that are `sufficiently old' so that their degrees are guaranteed to be `large enough' at time $2t$. More specifically, given a large enough constant $M>1$, we restrict our attention to vertices of $\mathbb{G}_{2t}$
that are born between times
\[T_1\coloneqq \log^M{t} \quad \text{and}\quad T_2\coloneqq t^{\frac{1-p}{2-p}}\Big(\log^{\frac{1}{2-p}}(t)f(t)\Big)^{-1}\ll t\]
whose degree is at least $Ct^{1/2}\log^{1/2}{t}$ for a large enough constant $C>0$, and show that such vertices are likely to connect one another during the time interval $[2t]\setminus[t]$. This is indeed the case because, by a union bound, the probability that these $\approx T_2$ `high degree' vertices do not form a clique is at most
\[T_2^2\prod_{s=t+1}^{2t}\Big(1-c_p \frac{C^2t\log(t)}{s^2}\Big)\]
for some constant $c_p>0$ (we recall that, during an edge step, we add one edge between two vertices selected independently with probability proportional to their degrees). But, since $1+x\leq e^x$ for every real $x$, we get
\[\prod_{s=t+1}^{2t}\Big(1-c_p \frac{C^2t\log(t)}{s^2}\Big)\lesssim \exp\Big(-c_p C^2t\log(t) \int_{t+1}^{2t}s^{-2}ds\Big)\approx \exp\big(-c_p C^2\log(t) \big)\]
and so, by taking a large enough $C$, we see that indeed the vertices of interest (i.e.\ those nodes of $\mathbb{G}_{2t}$ born in the time window $[T_1,T_2]$ with degree at least $Ct^{1/2}\log^{1/2}(t)$) form a clique of (the desired) size $\approx T_2$ with probability $1-o(1)$. Hence there remains to show that indeed a positive proportion of the vertices of $\mathbb{G}_{2t}$ born in the time interval $[T_1,T_2]$ have degree at least $Ct^{1/2}\log^{1/2}(t)$. To do so, let $X_t$ be the number of vertices in $\mathbb{G}_{2t}$ born in $[T_1,T_2]$ whose degree is less than $Ct^{1/2}\log^{1/2}(t)$. Then we must show that, for some constant $\delta>0$, $X_t\leq \delta T_2$ with high probability. Denoting by $\mathcal{L}_t$ the set of vertices of $\mathbb{G}_{2t}$ born in $[T_1,T_2]$, we can write
\begin{align*}
    \mathbb{E}[X_t]&=\sum_{u\in [2t]}^{}\mathbb{P}\big(\text{deg}(u,t)<Ct^{1/2}\log^{1/2}(t),u\in \mathcal{L}_t\big)\\
&=\sum_{u\in [2t]}^{}\mathbb{E}\Big[\mathbb{1}\{u\in \mathcal{L}_t\}\mathbb{P}\big(\text{deg}(u,t)<Ct^{1/2}\log^{1/2}(t)\mid b(u)\big)\Big],
\end{align*}
where $b(u)$ represents the \textit{birth time} of vertex $u\in \mathbb{G}_{2t}$. Suppose that, uniformly over all $u\in \mathcal{L}_t$, as $t\rightarrow \infty$ it holds that 
\begin{equation}\label{CCC}
    \mathbb{P}\big(\text{deg}(u,t)<Ct^{1/2}\log^{1/2}(t)\mid b(u)\big)=o(1).
\end{equation}
Then, since $|\mathcal{L}_t|\leq T_2$, we deduce that
\[\mathbb{E}[X_t]=o(T_2).\]
But then, by Markov's inequality, we get
\[\mathbb{P}(X_t>\delta T_2)\leq \frac{\mathbb{E}[X_t]}{\delta T_2}=o(1),\]
as desired. Hence there remains to establish (\ref{CCC}).
In order to show that the probability in (\ref{CCC}) is $o(1)$ (uniformly over the vertices under exam), we closely follow \cite{alves2021clustering} and analyze the time required by vertex $u$ to reach degree $k$. We do not give the details here, but a simple observation illustrates why the above should be true. Indeed, we observe that
\[\mathbb{E}[\text{deg}(u,t)\mid b(u)]\approx \Big(\frac{t}{b(u)}\Big)^{1-p/2};\]
also 
\[\Big(\frac{t}{b(u)}\Big)^{1-p/2}\gg t^{1/2}\log^{1/2}(t) \Longleftrightarrow b(u)\ll t^{\frac{1-p}{2-p}}/\log^{\frac{1}{2-p}}(t)\]
and the last statement holds because $b(u)\leq T_2$ for every $u\in \mathcal{L}_t$ and (by definition of $T_2$),
\[\frac{T_2}{t^{\frac{1-p}{2-p}}/\log^{\frac{1}{2-p}}(t)}=1/f(t)\ll 1.\]
%Concerning the choice of $T_1$, has been properly made in order to prove this last result. In particular, taking vertices born after $\log^M{(t)}$, let us have control on the degree growth, which turned out to be fundamental in order to obtain our result.

\section{Proof of Theorem \ref{mainthm}}
Define
\[L(t)=L(t,p,\varepsilon)\coloneqq \frac{t^{\frac{1-p}{2-p}}}{\log^{\frac{1}{2-p}}(t)f(t)},\]
where we recall that $f(t) \to \infty$ as $t \to \infty$.
In order to establish Theorem \ref{mainthm} it is sufficient to show that
\begin{equation}\label{goal}
    \mathbb{P}\left(\exists \mathcal{S}\subseteq V(\mathbb{G}_{2t}): |\mathcal{S}|\geq \frac{L(t)}{4} \text{ and }\{u,v\}\in E(\mathbb{G}_{2t}) \text{ }\forall u,v\in \mathcal{S} \text{ with }u\neq v\right)=1-o(1).
\end{equation}
For each vertex $u\in V(\mathbb{G}_{2t}) \subseteq [2t]$, let us write $b(u)$ for its (random) birth time. (Note that $b(u)\geq u$ for each vertex $u \in V(\mathbb{G}_{2t})$ since it is not guaranteed that a vertex is added at each step of the algorithm constructing $\mathbb{G}_{2t}$.) 
Let us denote by $\mathcal{L}_t$ the (random) subset of vertices of $\mathbb{G}_{2t}$ which are born between times $T_1=T_1(t,p)\in \mathbb{N}$ and $T_2=T_2(t,p)\in \mathbb{N}$, where (given any $M>0$)
\[\log^M(t)\ll T_1\leq \frac{T_2}{2}, \quad \text{ } 
T_2\coloneqq 2\lceil p^{-1} L(t)\rceil.\]
That is, 
\[\mathcal{L}_t\coloneqq \{u\in V(\mathbb{G}_{2t}):b(u)\in [T_1,T_2]\cap \mathbb{N}\}\subseteq V(\mathbb{G}_{2t}).\]
Note that, since we add at most one vertex at each step in the construction of the underlying random graph, we have that $|\mathcal{L}_t|\leq T_2-T_1$ deterministically. The probability in (\ref{goal}) is clearly at least
\begin{equation}\label{interstep}
    \mathbb{P}\left(\exists \mathcal{L}'_t\subseteq \mathcal{L}_t: |\mathcal{L}'_t|\geq \frac{|\mathcal{L}_t|}{2} \text{ and }\{u,v\}\in E(\mathbb{G}_{2t}) \text{ }\forall u,v\in \mathcal{L}'_t \text{ with }u\neq v, |\mathcal{L}_t|\geq \frac{L(t)}{2}\right).
\end{equation}
We define the counting random variable
\[X_t\coloneqq \sum_{u\in \mathcal{L}_t}^{}\mathbb{1}\{\text{deg}(u,t)<C_2t^{1/2}\log^{1/2}(t)\}\]
where $C_2=C_2(p)$ is some (large) constant to be selected later. In words, $X_t$ is just the number of vertices in $\mathcal{L}_t\subseteq \mathbb{G}_{2t}$ whose degree at time $t$ is strictly less than
\[C_2t^{1/2}\log^{1/2}(t)\eqqcolon h(t,C_2)=h.\]
On the event $\{X_t\leq |\mathcal{L}_t|/2-1\}$, there are at least $|\mathcal{L}_t|-X_t \geq |\mathcal{L}_t|/2+1 \geq \lceil |\mathcal{L}_t|/2\rceil$ vertices in $\mathcal{L}_t$ having degree at least $h$ at time $t$. Let us denote by $U_i$ the $i$-th (for $1\leq i\leq \lceil |\mathcal{L}_t|/2\rceil$) smallest vertex label in $\mathcal{L}_t$ with degree at least $h(t)$. That is,
\[U_1\coloneqq \min\{u\in \mathcal{L}_t:\text{deg}(u,t)\geq h\}\]
and, for $2\leq i\leq \left\lceil |\mathcal{L}_t|/2\right\rceil$,
\[U_i\coloneqq \min\{u\in \mathcal{L}_t\setminus \{U_1,\dots,U_{i-1}\}:\text{deg}(u,t)\geq h\}.\]
Then the probability in (\ref{interstep}) is at least
\begin{multline}\label{interstep2}
    \mathbb{P}\left(\{U_i,U_j\}\in E(\mathbb{G}_{2t}) \text{ }\forall 1\leq i,j\leq \left\lceil |\mathcal{L}_t|/2\right\rceil \text{ with }i\neq j, |\mathcal{L}_t|\geq \frac{L(t)}{2}, X_t\leq \frac{|\mathcal{L}_t|}{2}-1\right)\\
    \geq \mathbb{P}\left(\{U_i,U_j\}\in E(\mathbb{G}_{2t}) \text{ }\forall 1\leq i,j\leq \left\lceil |\mathcal{L}_t|/2\right\rceil \text{ with }i\neq j, |\mathcal{L}_t|\geq \frac{L(t)}{2}, X_t\leq \frac{L(t)}{8}\right),
\end{multline}
where the last inequality is due to the fact that, on the event $\{|\mathcal{L}_t|\geq L(t)/2\}$, we have $|\mathcal{L}_t|\gg 1$ and so 
\[\frac{|\mathcal{L}_t|}{2}-1\geq \frac{|\mathcal{L}_t|}{4}\geq \frac{L(t)}{8}.\]
The second probability in (\ref{interstep2}) is at least 
\begin{equation}\label{split}
    \mathbb{P}\left(\{U_i,U_j\}\in E(\mathbb{G}_{2t}) \text{ }\forall 1\leq i,j\leq \left\lceil |\mathcal{L}_t|/2\right\rceil \text{ with }i\neq j, X_t\leq \frac{L(t)}{8}\right)-\mathbb{P}\left(|\mathcal{L}_t|< \frac{L(t)}{2}\right).
\end{equation}
With the purpose of bounding the probability on the right-hand side of the above expression note that, since at each step in the construction of $\mathbb{G}_{2t}$ we add a vertex (independently of the previous steps) with probability $p$, we have $|\mathcal{L}_t|=_d\text{Bin}(T_2-T_1,p)$. Moreover, recalling  that $T_1\leq T_2/2$ by assumption, we get $(T_2-T_1)p\geq T_2p/2\geq L(t)$ and hence we see that $L(t)/2\leq p(T_2-T_1)/2$. Thus we can use Lemma \ref{concbinom} to write
\[
\mathbb{P}(|\mathcal{L}_t|<  L(t)/2) \leq \mathbb{P}(|\mathcal{L}_t|< p(T_2-T_1)/2)\leq \exp{(-p(T_2-T_1)/8)} \leq \exp\big(-\Omega(L(t))\big)=o(1).
\]
Since the event $\{X_t\leq L(t)/8\}$ belongs to $\mathcal{F}_t\coloneqq \sigma(\mathbb{G}_1,\dots,\mathbb{G}_{t})$, the sigma algebra generated by the graph process until time $t$, using the tower property of conditional expectation (and the previous estimate) we can bound from below the expression in (\ref{split}) by
\begin{equation}\label{towerprop}
    \mathbb{E}\left[\mathbb{1}\left\{ X_t\leq \frac{L(t)}{8}\right\}\mathbb{P}\left(\{U_i,U_j\}\in E(\mathbb{G}_{2t}) \text{ }\forall 1\leq i,j\leq \left\lceil |\mathcal{L}_t|/2\right\rceil \text{ with }i\neq j\mid \mathcal{F}_t\right)\right]-o(1).
\end{equation}
Note that for now we retain the event
\[\{ X_t\leq L(t)/8\}\subset \{ X_t\leq |\mathcal{L}_t|/2-1\}\]
because the $U_i$ are defined \textit{on} $\{X_t\leq |\mathcal{L}_t|/2-1\}$.
Recall that, by definition of $\mathcal{L}_t$, vertices $U_i$ and $U_j$ are born before time $T_2\ll t$. Thus $U_i$ and $U_j$ can be connected by an edge at time $s \geq t\gg   b(U_i) \vee b(U_j)$ if, and only if, at time $s$ an edge step occurs and both $U_i$ and $U_j$ are selected, and because
\[\text{deg}(U_i,t), \text{deg}(U_j,t) \geq h=C_2t^{1/2}\log^{1/2}(t),\]
we compute 
%\textcolor{red}{U: issue def of the $U_i$; they should be vertices already in the graph at time $t$!; then the above sentence should change to `...connected by an edge at time $s\geq t$...'}
\begin{align*}
    &\mathbb{P}\left(\{U_i,U_j\}\in E(\mathbb{G}_{2t})  \text{ }\forall 1\leq i,j\leq \left\lceil |\mathcal{L}_t|/2\right\rceil \text{ with }i\neq j\mid \mathcal{F}_t\right)\\ 
    &\geq 1-\sum_{i=1}^{\lceil |\mathcal{L}_t|/2\rceil}\sum_{j\neq i}^{\lceil |\mathcal{L}_t|/2\rceil}\mathbb{P}\Big(\{U_i,U_j\}\notin E(\mathbb{G}_{2t}) \text{ }\mid \mathcal{F}_t\Big)\\
    &\geq 1-\sum_{i=1}^{\lceil |\mathcal{L}_t|/2\rceil}\sum_{j\neq i}^{\lceil |\mathcal{L}_t|/2\rceil}\prod_{s=t+1}^{2t}\Big(1-2(1-p)\frac{\text{deg}(U_i,t)}{2s}\frac{\text{deg}(U_j,t)}{2s}\Big)\\
    &\geq 1-\sum_{i=1}^{\lceil |\mathcal{L}_t|/2\rceil}\sum_{j\neq i}^{\lceil |\mathcal{L}_t|/2\rceil}\prod_{s=t+1}^{2t}\Big(1-(1-p)\frac{C^2_2t\log(t)}{2s^2}\Big)\\
    &\geq 1-|\mathcal{L}_t|^2\exp\Big(-(1-p)\frac{C^2_2}{2}t\log(t)\sum_{s=t+1}^{2t}s^{-2}\Big),
\end{align*}
where for the last inequality we have used the standard bound $1+x\leq e^x$, which is valid for all $x\in \mathbb{R}$. Since
\[\sum_{s=t+1}^{2t}s^{-2}\geq \int_{t+1}^{2t+1}x^{-2} \geq \frac{1}{3t}\]
and $|\mathcal{L}_t|\leq T_2-T_1<T_2$ we obtain
\[1-|\mathcal{L}_t|^2\exp\Big(-(1-p)\frac{C^2_2}{2}t\log(t)\sum_{s=t+1}^{2t}s^{-2}\Big)\geq 1-T^2_2 \exp\Big(-(1-p)\frac{C^2_2}{6}\log(t)\Big)\]
and, for a large enough $C_2=C_2(p)>0$, %$>\sqrt{12/(2-p)}$ 
we see that
\[1-T^2_2 \exp\Big(-(1-p)\frac{C^2_2}{6}\log(t)\Big)=1-o(1).\]
Consequently, the expression in (\ref{towerprop}) is at least
\[(1-o(1))\mathbb{P}\left(X_t\leq \frac{L(t)}{8}\right)-o(1)=(1-o(1))\left[1-\mathbb{P}\left(X_t> \frac{L(t)}{8}\right)\right]-o(1).\]
Therefore, if we can show that
\begin{equation}\label{laststep}
   \mathbb{P}\left(X_t> \frac{L(t)}{8}\right)=o(1),
\end{equation}
then we conclude that (\ref{goal}) holds, and we are done. To establish (\ref{laststep}) we use Markov's inequality to bound
\[\mathbb{P}\left(X_t> \frac{L(t)}{8}\right)\leq \frac{8}{L(t)}\mathbb{E}[X_t].\]
Suppose that there exists $\delta_t=\delta_t(p)=o(1)$ such that 
\begin{equation}\label{mainbound}
    \mathbb{P}\big(\text{deg}(u,t)<h\mid b(u)\big)\leq \delta_t \quad \text{for every }u\in \mathcal{L}_t.
\end{equation}
Then
\begin{align*}
    \mathbb{E}[X_t]&=\sum_{u\in [2t]}^{}\mathbb{P}\big(\text{deg}(u,t)<h,u\in \mathcal{L}_t\big)\\
&=\sum_{u\in [2t]}^{}\mathbb{E}\Big[\mathbb{1}\{u\in \mathcal{L}_t\}\mathbb{P}\big(\text{deg}(u,t)<h\mid b(u)\big)\Big]\\
&\leq \delta_t\mathbb{E}[|\mathcal{L}_t|],
\end{align*}
where for the second equality, we have used the tower property and the fact that $u\in \mathcal{L}_t$ if, and only if, $b(u)$ is in $[T_1,T_2]\cap \mathbb{N}$ and the latter event is clearly fully determined by knowledge of $b(u)$. But then, since $\mathbb{E}[|\mathcal{L}_t|] = O(L(t))$, we get
\[\mathbb{P}\left(X_t> \frac{L(t)}{8}\right)\leq \frac{8}{L(t)}\mathbb{E}[X_t]=O(\delta_t)=o(1),\]
as desired. Therefore, in what follows, we focus on establishing (\ref{mainbound}). To this end, we first state the following elementary fact, whose proof is postponed to Section \ref{aux}. 

In order to simplify notation, we set 
\[\mathbb{P}_u(\cdot)\coloneqq \mathbb{P}(\cdot \mid b(u)),\]
the (conditional) probability given the birth time of $u$, where $u \in \mathcal{L}_t$ (so that $T_1 \leq b(u) \leq T_2 \ll t$).
\begin{lem}\label{upperbdeg}
    Given any $x\geq 4$ we have that
    \[\mathbb{P}_u\left(\text{deg}(u,t)\geq \left(\frac{t}{b(u)}\right)^{1-p/2}x\right)\leq \exp\left(-C_3\frac{x^2}{b(u)}\right),\]
    for some constant $C_3=C_3(p)>0$ which depends solely on $p$.
\end{lem}
With Lemma \ref{upperbdeg} at hand we can complete the proof of the theorem by establishing (\ref{mainbound}). The argument used here follows very closely the steps carried out in \cite{alves2021clustering}; however, since the relevant result of \cite{alves2021clustering} is stated for the \textit{sum} of the degrees of $m$ vertices, with $m$ a large enough constant, we preferred (for the sake of clarity and completeness) to include a complete proof, adapted to our current needs, rather than just stating the required modifications to the argument given in \cite{alves2021clustering}.
Recall that $u\in \mathcal{L}_t$ and so in particular $b(u)\geq T_1$.
Define
\[T^u_k\coloneqq \min\{i\geq 1:\text{deg}(u,i)=k\}, \text{ for }k\in \mathbb{N}.\]
In words, $T^u_k$ is just the first time at which the degree of vertex $u$ reaches $k$. Note that $T^u_k\geq b(u)+(k/2)$ almost surely. Indeed, at each step of the process the degree of $u$ can increase by at most $2$ (this occurs when there is an edge step and vertex $u$ is selected twice, thus forming a self-loop). Observe that $\text{deg}(u,t)<k$ if, and only if, $T^u_k>t$. Therefore, recalling that $h=C_2t^{1/2}\log^{1/2}(t)$ by definition, we see that
\begin{align}\label{firstiter}
    \nonumber\mathbb{P}_u\big(\text{deg}(u,t)< h\big)&\leq \mathbb{P}_u\big(\text{deg}(u,t)< \lceil h \rceil\big) \\
    \nonumber&=\mathbb{P}_u\big(T^u_{\lceil h \rceil}>t\big)\\
    &\leq \mathbb{P}_u\big(T^u_{\lceil h \rceil}>t,T^u_{\lceil h \rceil-1}\geq \lfloor g(h-1)\rfloor \big)+\mathbb{P}_u(T^u_{\lceil h \rceil-1}< \lfloor g(h-1)\rfloor),
\end{align}
where we set
\begin{equation}\label{defg}
    g(r)=g(r,b(u),p,\omega)\coloneqq \frac{r^{\frac{2}{2-p}}b(u)}{\omega} \vee \left(b(u)+\frac{r}{2}\right), \text{ for }1\leq r\leq \lceil h \rceil-1,
\end{equation}
with $\omega=\omega(t)\gg 1$ satisfying $\omega=O\left(\frac{b(u)}{\log(t)}\right)$. Note that since $b(u)\geq T_1$ (as $u \in \mathcal{L}_t$) and $T_1\gg \log(t)$, we can find such an $\omega$. The second term on the last inequality in \eqref{firstiter} is zero if 
\[\frac{(\lceil h \rceil-1)^{\frac{2}{2-p}}b(u)}{\omega}\leq b(u)+\frac{(\lceil h \rceil-1)}{2}\] 
exactly because $T^u_{h-1}\geq b(u)+(\lceil h \rceil-1)/2$.
Assuming this is not the case, we see that
\begin{align*}
\mathbb{P}_u(T^u_{\lceil h \rceil-1}< \lfloor g(h-1)\rfloor)&=\mathbb{P}_u(\text{deg}(u,\lfloor g(h-1)\rfloor)> \lceil h \rceil-1)
\\&=\mathbb{P}_u\left(\text{deg}\left(u,\left\lfloor (\lceil h \rceil-1)^{\frac{2}{2-p}}\frac{b(u)}{\omega}\right\rfloor\right)> \lceil h \rceil-1\right).
\end{align*}
Now setting $\ell -1 \coloneqq \left\lfloor (\lceil h \rceil-1)^{\frac{2}{2-p}}\frac{b(u)}{\omega}\right\rfloor$, wee see that:
\begin{itemize}
    \item if $(\lceil h \rceil-1)^{\frac{2}{2-p}}b(u)/\omega$ is an integer, then solving for $\lceil h \rceil-1$ we obtain
    \[\lceil h \rceil-1=\left((\ell -1)\omega/b(u)\right)^{1-p/2};\]
    \item if $(\lceil h \rceil-1)^{\frac{2}{2-p}}b(u)/\omega$ is not an integer, then solving for $\lceil h \rceil-1$ we obtain
    \[\lceil h \rceil-1\geq \Big((\ell -1)\omega/b(u)\Big)^{1-p/2}.\]
\end{itemize}
Therefore in either case we arrive at
\begin{align*}
    \mathbb{P}_u\big(\text{deg}(u,\lfloor (\lceil h \rceil-1)^{\frac{2}{2-p}}b(u)/\omega\rfloor)> \lceil h \rceil-1\big)&\leq \mathbb{P}_u\big(\text{deg}(u,\ell -1)>\Big((\ell -1)/b(u)\Big)^{1-p/2}\omega^{1-p/2}\Big)\\
    &\leq \exp\left(-C_3\frac{\omega^{2-p}}{b(u)}\right),
\end{align*}
where the last inequality is due to Lemma \ref{upperbdeg}. Thus, going back to (\ref{firstiter}), we obtain
\[\mathbb{P}_u\big(\text{deg}(u,t)< h\big)\leq \mathbb{P}_u\big(T^u_{\lceil h \rceil}>t,T^u_{\lceil h \rceil-1}\geq \lfloor g(h-1)\rfloor \big)+\exp\left(-C_3\frac{\omega^{2-p}}{b(u)}\right).\]
The first term on the right-hand side is equal to
\begin{equation}\label{sum}
    \sum_{s\geq \lfloor g(h-1)\rfloor}\mathbb{P}_u(T^u_{\lceil h \rceil}>t\mid T^u_{\lceil h \rceil-1}=s)\mathbb{P}_u(T^u_{\lceil h \rceil-1}=s).
\end{equation}
Using the definition of the model, setting
\[\Delta \text{deg}(u,t+1)\coloneqq  \text{deg}(u,t+1)-\text{deg}(u,t),\]
the increment of the degree of vertex $u$ at time $t$, we have that
\begin{align*}
\mathbb{P}_u(\Delta \text{deg}(u,t+1) \geq 1 | \mathcal{F}_{t}) &= p\frac{\text{deg}(u,t)}{2t} + (1-p)\frac{2\text{deg}(u,t)}{2t} - (1-p) \frac{\text{deg}(u,t)^2}{4t^2} \\
&= \left(1-\frac{p}{2}\right)\frac{\text{deg}(u,t)}{t} - (1-p)\frac{\text{deg}(u,t)^2}{4t^2}.
\end{align*}
From this we get
\begin{align*}
    \mathbb{P}_u(T^u_{\lceil h \rceil}>t\mid T^u_{\lceil h \rceil-1}=s)&\leq \prod_{r=s}^{t}\Big(1-\Big(\left(1-\frac{p}{2}\right)\frac{\lceil h \rceil-1}{r}-(1-p)\frac{(\lceil h \rceil-1)^2}{4r^2}\Big)\Big)\\
    &\leq \prod_{r=s}^{t}\Big(1-\Big(\left(1-\frac{p}{2}\right)\frac{\lceil h \rceil-1}{r}-(1-p)\frac{(\lceil h \rceil-1)^2}{4r\lfloor g(h-1)\rfloor}\Big)\Big)\\
    &=\prod_{r=s}^{t}\Big[1-\Big(\left(1-\frac{p}{2}\right)\frac{\lceil h \rceil-1}{r}\Big(1-\frac{(1-p)(\lceil h \rceil-1)}{2(2-p)\lfloor g(h-1)\rfloor}\Big)\Big)\Big],
\end{align*}
where for the second inequality we have used that $r\geq s\geq \lfloor g(h-1)\rfloor$ in the sum (\ref{sum}). Setting 
\begin{equation}\label{defdelta}
    \delta_p(r)=\delta_p(r,u,\omega)\coloneqq \frac{(1-p)r}{2(2-p)\lfloor g(r)\rfloor} \text{ for }1\leq r\leq \lceil h \rceil-1,
\end{equation}
and using again $(1+x) \leq e^x$, and that for $r$ large enough holds that $e^{-\frac{a}{r}} \leq \big(\frac{r}{r+1}\big)^{a}$ we have:
\begin{align*}
    1-\left(1-\frac{p}{2}\right)\frac{\lceil h \rceil-1}{r}\big(1-\delta_p(\lceil h \rceil-1)\big)&\leq \big(e^{\frac{1}{r}}\big)^{-\left(1-\frac{p}{2}\right)(\lceil h \rceil-1)(1-\delta_p(\lceil h \rceil-1))}\\&
    \leq \Big(\frac{r}{r+1}\Big)^{\left(1-\frac{p}{2}\right)(\lceil h \rceil-1)(1-\delta_p(\lceil h \rceil-1))},
\end{align*}
we arrive at
\begin{align} \label{controlwithexp}
    \nonumber
    \mathbb{P}_u(T^u_{\lceil h \rceil}>t\mid T^u_{\lceil h \rceil-1}=s)&\leq \prod_{r=s}^{t}\Big(\frac{r}{r+1}\Big)^{\left(1-\frac{p}{2}\right)(\lceil h \rceil-1)(1-\delta_p(\lceil h \rceil-1))}\\
    &\leq \Big(\frac{s}{t}\Big)^{\left(1-\frac{p}{2}\right)(\lceil h \rceil-1)(1-\delta_p(\lceil h \rceil-1))}.
\end{align}
Let $\eta_{h-1} =_d \text{Exp}\left(\left(1-\frac{p}{2}\right)\left(\lceil h \rceil-1\right)\left(1-\delta_p\left(\lceil h \rceil-1\right)\right)\right)$, independent of everything else. Then it follows from \eqref{controlwithexp}
\[\mathbb{P}_u\left(T^u_{\lceil h \rceil}>t\mid T^u_{\lceil h \rceil-1}=s\right)\leq \mathbb{P}\left(\exp(\eta_{\lceil h \rceil-1})>\frac{t}{s}\right)=\mathbb{P}_u\left(\exp(\eta_{\lceil h \rceil-1})>\frac{t}{s}\right),\]
where the last equality due to the independence of $\eta_{\lceil h \rceil-1}$ and the random graph. Therefore the sum in (\ref{sum}) is at most
\[\sum_{s\geq \lfloor g(h-1)\rfloor}\mathbb{P}_u\left(\exp(\eta_{\lceil h \rceil-1})>\frac{t}{s}\right)\mathbb{P}_u(T^u_{\lceil h \rceil-1}=s)\leq \mathbb{P}_u\big(T^u_{\lceil h \rceil-1}\exp(\eta_{h-1})>t).\]  
We now repeat the argument in \eqref{firstiter} to get
\begin{align*}
    \mathbb{P}_u(T^u_{\lceil h\rceil-1} \exp{(\eta_{\lceil h \rceil -1})} > t) \leq &\mathbb{P}_u(T^u_{\lceil h\rceil-1} \exp{(\eta_{\lceil h \rceil -1})} > t, T^u_{\lceil h \rceil -2} \geq \lfloor g(h-2)\rfloor) \\&+ \mathbb{P}_u(T^u_{\lceil h \rceil-2} <\lfloor g(h-2) \rfloor)
\end{align*}
and so by repeating the same calculation, we have that
\[\mathbb{P}_u\big(T^u_{\lceil h \rceil-1}\exp(\eta_{h-1})>t) \leq \mathbb{P}_u\big(T^u_{\lceil h \rceil-1}\exp(\eta_{\lceil h \rceil-1}+\eta_{\lceil h \rceil -2})>t\big)+\exp\left(-C_3\frac{\omega^{2-p}}{b(u)}\right).\]
Iterating we arrive at
\begin{equation}\label{gobackhere}
    \mathbb{P}_u\big(T^u_{\lceil h \rceil}>t\big)\leq \mathbb{P}_u\left(T^u_{1}\exp\left(\sum_{r=1}^{\lceil h \rceil -1}\eta_r\right)>t\right)+h\exp\left(-C_3\frac{\omega^{2-p}}{b(u)}\right),
\end{equation}
where each $\eta_{r} =_d \text{Exp}\left(\left(1-\frac{p}{2}\right)r(1-\delta_p(r))\right)$, which we take to be independent of $\eta_i$ for $i\neq r$ (and of the random graph).
Clearly $T^u_1=b(u)$ and hence 
\begin{align*}
    \mathbb{P}_u\left(T^u_{1}\exp\left(\sum_{r=1}^{\lceil h \rceil -1}\eta_r\right)>t\right)=\mathbb{P}_u\left(\exp\left(\sum_{r=1}^{\lceil h \rceil -1}\eta_r\right)>\frac{t}{b(u)}\right)&\leq \frac{b(u)^{\lambda}}{t^{\lambda}}\mathbb{E}\left[\exp\left(\lambda \sum_{r=1}^{\lceil h \rceil -1}\eta_r\right)\right]\\
    &=\frac{b(u)^{\lambda}}{t^{\lambda}}\prod_{r=1}^{\lceil h \rceil -1}\mathbb{E}\big[e^{\lambda \eta_r}\big],
\end{align*}
for every $\lambda>0$, where we have used Markov's inequality and the fact that the $\eta_r$ are independent of one another (and of the random graph). Now, in order to obtain a meaningful expression on the right-hand side of the last inequality, we need $\lambda$ to be smaller than each parameter of the $\eta_r$. Since for $1\leq r\leq \lceil h \rceil-1$ we have that
\[\left(1-\frac{p}{2}\right)r(1-\delta_p(r))\geq \left(1-\frac{p}{2}\right)(1-\delta_p(r))\geq \left(1-\frac{p}{2}\right)(1-\delta^+_p),\]
with $\delta^+_p\coloneqq \max\{\delta_p(r):1\leq r\leq \lceil h \rceil -1\}$, we take $\lambda\in \left(0,\left(1-\frac{p}{2}\right)(1-\delta^+_p)\right)$. Using (\ref{defg}) and (\ref{defdelta}) we further observe that (as $\omega\ll b(u)$ by hypothesis) for each $1\leq r\leq \lceil h \rceil -1$
\begin{equation}\label{maxdelta}
    \delta_p(r)\leq  \frac{(1-p)r}{2(2-p)\left\lfloor r^{\frac{2}{2-p}}\frac{b(u)}{\omega}\right\rfloor}\leq \frac{(1-p)\omega}{(2-p) r^{\frac{p}{2-p}}b(u)}\leq \frac{(1-p)\omega}{(2-p) b(u)}\ll 1,
\end{equation}
whence $\delta^+_p\ll 1$ too. We now write
\[\prod_{r=1}^{\lceil h \rceil -1}\mathbb{E}\big[e^{\lambda \eta_r}\big]=\prod_{r=1}^{\lceil h \rceil -1}\left(1-\frac{\lambda}{\left(1-\frac{p}{2}\right)r(1-\delta_p(r))}\right)^{-1}.\]
Since it's possible to show that $\log(1-y)>-y-y^2$ for $0<y<0.69$, by taking a small enough $\lambda$ (e.g. $\lambda<0.5$ suffices) we see that
\begin{align*}
    \left(1-\frac{\lambda}{\left(1-\frac{p}{2}\right)r(1-\delta_p(r))}\right)^{-1}&=\exp\left(-\log\left(1-\frac{\lambda}{\left(1-\frac{p}{2}\right)r(1-\delta_p(r))}\right)\right)\\
    &\leq \exp\left(\frac{\lambda}{\left(1-\frac{p}{2}\right)r(1-\delta_p(r))}+\frac{\lambda^2}{\left[\left(1-\frac{p}{2}\right)r(1-\delta_p(r))\right]^2}\right)
\end{align*}
so that then
\begin{align*}
    \prod_{r=1}^{\lceil h \rceil -1}\mathbb{E}\big[e^{\lambda \eta_r}\big]\leq C(\lambda,p)\exp\left(\frac{\lambda}{1-\frac{p}{2}}\frac{1}{1-\delta^+_p}\sum_{r=1}^{\lceil h \rceil -1}r^{-1}\right)&\leq C(\lambda,p)\exp\left(\frac{\lambda}{1-\frac{p}{2}}\frac{\log(\lceil h \rceil )}{1-\delta^+_p}\right)\\
    &=C(\lambda,p)\lceil h \rceil ^{\frac{\lambda}{1-\frac{p}{2}}(1-\delta^+_p)^{-1}},
\end{align*}
Whence
\[\frac{b(u)^{\lambda}}{t^{\lambda}}\prod_{r=1}^{\lceil h \rceil -1}\mathbb{E}\big[e^{\lambda \eta_r}\big]\leq C(\lambda,p)\frac{b(u)^{\lambda}}{t^{\lambda}}\lceil h \rceil ^{\frac{\lambda}{1-\frac{p}{2}}(1-\delta^+_p)^{-1}}.\]
Using the definition of $T_2$ together with the fact that $\lceil h \rceil=O(t^{1/2}\log(t))$ (where the constant in the $O(\cdot)$ notation depends solely on $p$), we see that the expression on the right-hand side of the last inequality is at most a constant (depending on $p$) times
\begin{equation}\label{mainexpr}
    t^{-\lambda \big(1-\frac{1-p}{2-p}\big)}\big(\log(t)\big)^{-\lambda\big(\frac{1}{2-p}\big)} f(t)^{-\lambda}\cdot t^{\frac{\lambda}{2-p}(1-\delta^+_p)^{-1}}\big(\log(t)\big)^{\frac{\lambda}{2-p}(1-\delta^+_p)^{-1}}.
\end{equation}
Since $\delta^+_p\ll 1$, we can write
\[(1-\delta^+_p)^{-1}=1+O(\delta^+_p)\]
and hence (for large enough $t$)
\[t^{\frac{\lambda}{2-p}(1-\delta^+_p)^{-1}}\big(\log(t)\big)^{\frac{\lambda}{2-p}(1-\delta^+_p)^{-1}}\leq t^{\frac{\lambda}{2-p}}\big(\log(t)\big)^{\frac{\lambda}{2-p}}t^{O(\delta^+_p)}.\]
(The constant in the $O(\delta^+_p)$ notation now depends on $\lambda$ and $p$, but this is of no importance as both quantities are of constant order.) Hence the expression in (\ref{mainexpr}) is at most
\[f(t)^{-\lambda}t^{O(\delta^+_p)}.\]
Finally, we notice from (\ref{maxdelta}) that
\[t^{O(\delta^+_p)}=\exp\Big(O(\delta^+_p)\log(t)\Big)=\exp\left(O\left(\frac{\omega}{b(u)}\log(t)\right)\right)=O(1)\]
since $\omega=O\left(\frac{b(u)}{\log(t)}\right)$ by assumption. In summary, we have shown that
\[\mathbb{P}_u\left(T^u_{1}\exp\left(\sum_{r=1}^{\lceil h \rceil -1}\eta_r\right)>t\right)\leq C(\lambda,p)f(t)^{-\lambda}\]
for some constant $C(\lambda,p)>0$ which depends solely on $\lambda,p$. Thus, going back to (\ref{gobackhere}), we conclude that
\[\mathbb{P}_u\big(T^u_{\lceil h \rceil}>t\big)\leq C(\lambda,p)f(t)^{-\lambda}+h\exp\left(-C_3\frac{\omega^{2-p}}{b(u)}\right).\]
Taking e.g. $\omega \coloneqq (C_3b(u)\log{(t)})^{1/(2-p)}$ for a large enough $C_3 = C_3(p)>0$ (note that this choice respects the requirements $1 \ll w = O(b(u)/\log{(t)})$, as $b(u) \geq T_1$ and $T_1$ grows asymptotically faster than any power of $\log{t}$ by assumption), we can conclude that
\[\mathbb{P}_u\big(T^u_{\lceil h \rceil}>t\big)\leq 2C(\lambda,p)f(t)^{-\lambda}\eqqcolon \delta_t(p),\]
which then establishes (\ref{mainbound}) (upon fixing a value of $\lambda\in (0,1/2)$), since $f(t)$ diverges with $t$.
\subsection{Proof of Lemma \ref{upperbdeg}}\label{aux}
There remains to establish an upper bound for the (conditional) probability that the degree of a given node $u$ is `too large', given its birth time $b(u)$.
Recall that we are taking $u \in \mathcal{L}_t$, which means that $T_1 \leq b(u) \leq T_2\ll t$.
\begin{proof}[Proof of Lemma \ref{upperbdeg}]
The proof is quite similar as the one given in \cite{alves2021clustering}, which in turn uses a standard argument based on Azuma-Hoeffding inequality stated in Lemma \ref{concmg}. First of all we note that, under $\mathbb{P}_u$, the process
\[M_u(t)\coloneqq \frac{\text{deg}(u,t)}{\prod_{s=b(u)}^{t-1}\Big(1+\frac{1-p/2}{s}\Big)}, \quad t\geq b(u),\]
defines a martingale started at $1$ (we use the convention that empty products equal $1$). Using that $\log(1+x)\leq x$ for every $x>-1$ and 
\[\sum_{s=b(u)}^{t-1}s^{-1}\leq \int_{b(u)}^{t}x^{-1}dx=\log\left(\frac{t}{b(u)}\right),\]
we get
\[\prod_{s=b(u)}^{t-1}\left(1+\frac{1-p/2}{s}\right)=\exp\left(\sum_{s=b(u)}^{t-1}\log\left(1+\frac{1-p/2}{s}\right)\right)\leq \left(\frac{t}{b(u)}\right)^{1-p/2}.\]
Hence we obtain
    \[\mathbb{P}_u\left(\text{deg}(u,t)\geq \left(\frac{t}{b(u)}\right)^{1-p/2}x\right)\leq \mathbb{P}_u\left(M_u(t)>x\right).\]
Next note that for $i>b(u)$ :
\begin{align*}
|M_u(i)-M_u(i-1)|&= \left| \frac{\text{deg}(u,i)}{\prod_{s=b(u)}^{i-1}\left(1+\frac{1-p/2}{s}\right)} - \frac{\text{deg}(u,i-1)}{\prod_{s=b(u)}^{i-2}\left(1+\frac{1-p/2}{s}\right)} \right|\\ &= \frac{1}{\prod_{s=b(u)}^{i-1}\left(1+\frac{1-p/2}{s}\right)} \left|\text{deg}(u,i)-\text{deg}(u,i-1)\left(1+\frac{1-p/2}{i-1}\right)\right| \\&\leq \frac{3-p/2}{\prod_{s=b(u)}^{i-1}\left(1+\frac{1-p/2}{s}\right)}\leq \frac{6}{\left(i/b(u)\right)^{1-p/2}},
\end{align*}
where for the second last inequality we used that $\text{deg}(u,i) \leq \text{deg}(u,i-1)+2$ 
while for the last inequality we have used that when $t$ is large enough (as $\log(1+x)\geq x-x^2$ for $x\geq 0$)
\begin{align*}
    \prod_{s=b(u)}^{t-1}\left(1+\frac{1-p/2}{s}\right)&=\exp\left(\sum_{s=b(u)}^{t-1}\log\left(1+\frac{1-p/2}{s}\right)\right)\\
    &\geq \left(\frac{t}{b(u)}\right)^{1-p/2}(1-o(1))\\
    &\geq2^{-1}\left(\frac{t}{b(u)}\right)^{1-p/2}.
\end{align*}
Since $M_u(b(u))=1$ and $x\geq 4$, we can use Lemma \ref{concmg} to bound 
\[\mathbb{P}_u\left(M_u(t)>x\right)\leq \mathbb{P}_u\left(M_u(t)-M_u(b(u))>\frac{x}{2}\right)\leq \exp\left(-\frac{x^2/4}{2\sum_{s=b(u)}^{t}36/\left(s/b(u)\right)^{2-p}}\right).\]
Since 
\[2\sum_{s=b(u)}^{t-1}\frac{36}{\left(\frac{s}{b(u)}\right)^{2-p}}\leq 72 (1-p)^{-1} 2^{1-p}  b(u),\]
we arrive at 
    \[\mathbb{P}_u\left(\text{deg}(u,t)\geq \left(\frac{t}{b(u)}\right)^{1-\frac{p}{2}}x\right)\leq\exp\left(-\frac{2^{-(6-p)}(1-p)}{9}\frac{x^2}{b(u)}\right)\]
    thus completing the proof.
\end{proof}

\section*{Acknowledgement} Both authors thank Rongfeng Sun and the National University of Singapore for their financial support to CDA.

\bibliographystyle{unsrt}
\bibliography{bibliography}
\end{document}